\documentclass[12pt,reqno]{amsart}

\newcommand\version{March 27, 2009}

\usepackage{amsmath,amsfonts,amsthm,amssymb,amsxtra}

\newtheorem{theorem}{Theorem}[section]
\newtheorem{proposition}[theorem]{Proposition}
\newtheorem{lemma}[theorem]{Lemma}

\theoremstyle{definition}

\newtheorem{assumption}[theorem]{Assumption}

\theoremstyle{remark}

\newtheorem{remark}[theorem]{Remark}

\numberwithin{equation}{section}

\newcommand{\C}{\mathbb{C}}

\newcommand{\const}{\mathrm{const}\ }

\renewcommand{\epsilon}{\varepsilon}

\newcommand{\loc}{{\rm loc}}

\newcommand{\N}{\mathbb{N}}

\renewcommand{\phi}{\varphi}
\newcommand{\R}{\mathbb{R}}

\newcommand{\Z}{\mathbb{Z}}

\DeclareMathOperator{\curl}{curl}
\DeclareMathOperator{\rmd}{d}

\DeclareMathOperator{\dom}{dom}
\DeclareMathOperator{\rme}{e}

\DeclareMathOperator{\re}{Re}
\DeclareMathOperator{\spec}{spec}

\begin{document}

\title[Weakly coupled bound states of Pauli operators --- \version]{Weakly coupled bound states of Pauli operators}

\author{Rupert L. Frank}
\address{Rupert L. Frank, Department of Mathematics,
Princeton University, Washington Road, Princeton, NJ 08544, USA}
\email{rlfrank@math.princeton.edu}

\author{Sergey Morozov}
\address{Sergey Morozov, Department of Mathematics, University College London, Gower Street, London, WC1E 6BT, UK}
\email{morozov@math.ucl.ac.uk}

\author{Semjon Vugalter}
\address{Semjon Vugalter, Mathematisches Institut der Ludwig--Maximilians--Universit\"at, Theresienstr. 39, 80333 Munich, Germany}
\email{wugalter@math.lmu.de}

\thanks{\copyright\, 2009 by the authors. This paper may be reproduced, in its entirety, for non-commercial purposes.}

\begin{abstract}
We consider the two-dimensional Pauli operator perturbed by a weakly coupled, attractive potential. We show that besides the eigenvalues arising from the Aharonov-Casher zero modes there are two or one (depending on whether the flux of the magnetic field is integer or not) additional eigenvalues for arbitrarily small coupling and we calculate their asymptotics in the weak coupling limit.
\end{abstract}

\maketitle

\section{Introduction and main results}

\subsection{Introduction}

This paper is concerned with negative eigenvalues of perturbations of the two-dimensional Pauli operator
\begin{equation}\label{eq:pauli}
P := \big( {\bf \sigma}\cdot (-i\nabla +A) \big)^2
\qquad \text{in} \ L_2(\R^2,\C^2) \,.
\end{equation}
Here ${\bf \sigma} = (\sigma_1,\sigma_2)$ is the pair of the first two Pauli matrices
\begin{equation*}
  \sigma_1 =   \begin{pmatrix}
    0& 1\\
    1& 0
  \end{pmatrix}\,, \quad
  \sigma_2 =  \begin{pmatrix}
    0& -i\\
    i& 0
  \end{pmatrix}\,,
 \end{equation*}
and $A$ is a real vector potential corresponding to the magnetic field $B=\curl A$. Of course, $P$ is non-negative, and it is well-known that the point $0$ may be an eigenvalue of $P$. Indeed, the Aharonov-Casher theorem (see, e.g., \cite[Thm. 6.5]{CFKS}) asserts that if $B$ is, say, bounded with compact support, then the dimension of the kernel of $P$ is 
\begin{equation}\label{eq:n}
N:= \#\big\{m\in\N_0 : m<|\Phi|-1 \big\} \,,
\end{equation}
where
\begin{equation}\label{eq:flux}
\Phi := \frac1{2\pi} \int_{\R^2} B(x) \,\rmd x
\end{equation}
is the total flux of $B$. It is less known that $P$ also has a virtual level at $0$. Indeed, it was shown by Weidl \cite{W} that if $V$ is a non-negative, sufficiently regular and not identically zero function, then for all sufficiently small $\alpha>0$ the perturbed Pauli operator $P-\alpha V$ has exactly 
\begin{equation}
\label{eq:n'}
N':=
\begin{cases}
N + 1 & \text{if}\ \Phi\in \R\setminus\Z \,, \\
N + 2 & \text{if}\ \Phi\in \Z
\end{cases}
\end{equation}
negative eigenvalues. We express this fact by saying that $P$ has, in addition to its $N$ eigenvalues, one, resp. two virtual levels at zero.

These `additional' eigenvalues are of physical interest, in particular, since an anomalous magnetic moment $g>2$ corresponds to a perturbation $\alpha V = (g-2) B$ of the Pauli operator. Note that $g=2.0023$ for an electron. We refer to \cite{BCEZ} and references therein for more on this.

Let $\lambda_1(\alpha), \dots, \lambda_{N'}(\alpha)$ denote the $N'$ smallest eigenvalues of the operator $P-\alpha V$ in non-decreasing order. Throughout we assume that $B$ is compactly supported and radially symmetric. The goal of this paper is to obtain the asymptotic behavior of these eigenvalues as $\alpha\to 0+$. While it follows in a rather straightforward manner that
$$
\lambda_j(\alpha) \sim - c_j \alpha \,,
\qquad j=1,\ldots,N\,,
$$
as $\alpha\to 0+$, our main result is that
$$
\lambda_{N+1}(\alpha) \sim - c_{N+1} \alpha^{1/\mu}
\qquad
\text{if}\ \mu:= \big\{|\Phi|\big\} = |\Phi|-N \in(0,1)
$$
and
$$
\lambda_{N+1}(\alpha) \sim - \frac{c_{N+1} \alpha}{\big|\ln(c_{N+1} \alpha)\big|}\,,
\quad
\ln\big|\lambda_{N+2}(\alpha)\big| \sim -\frac 1 {c_{N+2}\alpha} 
\quad
\text{if}\  |\Phi|=N\in\N \,.
$$
(The result is slightly different in the case $\Phi=0$ and we refer to Theorem \ref{mainint} below for the precise statement.) Moreover, we obtain explicit expressions for the coefficients $c_j>0$ as well as estimates for the remainders in the above asymptotic expressions.

We note that our result quantifies a paramagnetic effect of the Pauli operator. Indeed, while the ground state energy of the Schr\"odinger operator $-\Delta -\alpha V$ is exponentially small in $\alpha$ in the weak coupling limit (see \cite{S1}), the addition of an arbitrarily small magnetic field with non-zero flux $\Phi$ leads to a much more negative ground state energy of the Pauli operator $P-\alpha V$ which is of the order $\alpha^{1/|\Phi|}$ if $|\Phi|<1$, $\alpha |\ln\alpha|^{-1}$ if $|\Phi|=1$ and $\alpha$ if $|\Phi|>1$.

\bigskip

The existence of weakly coupled eigenvalues can intuitively be understood as follows. Introducing the function
\begin{equation}\label{xi}
\xi(x):= -(2\pi)^{-1}\int_{\mathbb{R}^2}B(y)\ln|x- y| \,\rmd y \,,
\end{equation}
one easily finds the well-known relation
\begin{equation}
 \int_{\R^2} \big| \mathbf{\sigma}\cdot(-i\nabla+A)\psi\big|^2 \,\rmd x 
= 4 \int_{\R^2} \big( \rme^{-2\xi} |\partial_{\overline z} \rme^\xi \psi_+|^2 + \rme^{2\xi} |\partial_{z} \rme^{-\xi} \psi_-|^2 \big) \,\rmd x \,,
\end{equation}
where $\psi=(\psi^+,\psi^-)$, $\partial_z=\frac12(\partial_{x_1} -i \partial_{x_2})$ and $\partial_{\overline z}=\frac12(\partial_{x_1} +i \partial_{x_2})$. This suggests that zero modes (i.e., solutions of the equation $P\psi=0$) should be of the form $(\Omega_k^+,0)$ and $(0,\Omega_k^-)$, where
\begin{equation}\label{eq:omega}
\Omega_{k}^\pm(x) := (2\pi)^{-1/2} \rme^{\mp\xi(x)} (x_1 \pm ix_2)^k \,, \qquad k\in\N_0 \,.
\end{equation}
For the sake of definiteness let us assume that $\Phi>0$. If $B$ is radial and compactly supported, then $\xi(x) = - \Phi \ln|x|$ for large $|x|$ by Newton's theorem and hence all the $\Omega_{k}^+$ are increasing at infinity and do not belong to $L_2(\R^2)$. In contrast, $\Omega_k^-$ belongs to $L_2(\R^2)$ iff $0\leqslant k<\Phi-1$. Those $\Omega_k^-$ give rise to eigenvalues of $P-\alpha V$ which disappear \emph{linearly} in the weak coupling limit. What is more important for us is that the functions $\Omega_k^-$ with $\Phi-1\leqslant k\leqslant \Phi$ are \emph{bounded} (although they do not belong to $L_2(\R^2)$). It was already observed in \cite{BCEZ} that any perturbation by a negative potential $-V$ will turn these functions into $L_2$-eigenfunctions. Our key point is that \emph{the weak coupling asymptotics of the eigenvalues are determined by the spatial asymptotics of the functions $\Omega_k^-$}.

\bigskip

There is an enormous literature on weakly coupled eigenvalues and low energy behavior of Schr\"odinger operators from which we only mention \cite{S1,BGS,K,S2,KS,W}, the surveys \cite{B,P} and the recent papers \cite{AZ1,AZ2,Ko,HS}. We emphasize that techniques from weakly coupled Schr\"odinger operators have also turned out to be useful in a non-linear context \cite{FHNS}.

All papers on weak coupling asymptotics, which we are aware of, are based on the Birman-Schwinger principle and on operator-theoretic arguments. They rely upon very detailed knowledge of the unperturbed Green's function which is, of course, explicitly known for the Laplacian. It seems very unlikely that such information can be obtained in the generality in which we work here.

Instead we propose a completely different, purely variational approach. Its advantage is that it closely follows the above mentioned intuition that weak coupling asymptotics are determined by spatial asymptotics of resonance functions, which is often obscured in the operator-theoretic approach. It is by no means restricted to the problem under consideration and it allows one to recover in a simple way and improve upon the results in \cite{S1,BGS,K,KS,Ko} concerning the lowest eigenvalue. While the ground state energy is of primary physical interest, we should say that we do not see in general how to obtain results on excited states with our method. This is where in the present situation the radial symmetry of the problem comes in, which reduces the problem to \emph{ground state problems} for half-line operators. We emphasize, however, that our techniques are not restricted to one-dimensional problems.

\subsection*{Acknowledgements}
The authors would like to thank T. Weidl for drawing their attention to this problem and for helpful discussions. R.F. has profitted from discussions with T. Ekholm and H. Kovarik. The authors gratefully acknowledge the hospitality at Stuttgart University, KTH Stockholm and ESI Vienna, where parts of this work were done. This work was partially supported through Deutsche Forschungsgemeinschaft's (DFG) grant FR 2664/1-1, US National Science Foundation's grant PHY 06 52854 (R.F.), DFG grant SI 348/12-2, Engineering and Physical Sciences Research Council's grant EP/F029721/1 (S.M.) and DFG grant WE 1964/2 (S.V.).


\subsection{Main results}
\label{sec:main}

Let us state the precise conditions on the magnetic field and the electric potential. 

\begin{assumption}\label{ass:bv}
Let $B$ and $V$ be radial, real-valued and measurable functions with compact support in $\R^2$ such that
\begin{equation}
\label{eq:assbv}
\int_{\R^2} |B| \big(1+\ln_-|x|\big) \,\rmd x < \infty
\quad \text{and} \quad
\int_{\R^2} |V| \big(1+\ln_-|x|\big) \,\rmd x < \infty \,.
\end{equation}
\end{assumption}

Here and below $t_\pm\!:=\max\{\pm t,0\}$ for a number or a function $t$.
Let $A\in L_{2,\loc}(\R^2,\R^2)$ be a vector field with $\curl A= B$ (see \eqref{vector potential} for a convenient explicit choice). Under Assumption \ref{ass:bv} the operators $P-\alpha V$, $\alpha\in\R$, are defined through the closure of the quadratic forms
$$
\int_{\R^2} \big| \mathbf{\sigma}\cdot(-i\nabla +A)\psi\big|^2 \,\rmd x - \alpha \int_{\R^2} V |\psi|^2 \,\rmd x \,,
\qquad \psi\in C_0^\infty(\R^2,\C^2) \,.
$$
Using $\Phi$, $N$, $N'$ and $\Omega_k^\pm$ from \eqref{eq:flux}, \eqref{eq:n}, \eqref{eq:n'} and \eqref{eq:omega} we define
\begin{align}
 \label{eq:vm}
v_{k} & := \int_{\R^2} V |\Omega_{k}^\mp|^2 \,\rmd x \,,
\qquad \text{if} \ \pm\Phi>0\ \text{and} \ k=0,\ldots, N'-1 \,, \\
v_0^\pm & := \int_{\R^2} V |\Omega_{0}^\pm|^2 \,\rmd x \,,
\qquad \text{if} \ \Phi=0 \,,
\end{align}
and $m_k := \int_{\R^2} |\Omega_{k}^\mp|^2 \,\rmd x$ if $\pm\Phi>0$ and $k=0,\ldots,N-1$.
The main result of this paper are the following two theorems concerning the weak coupling asymptotics for Pauli operators.

\begin{theorem}[Case of noninteger magnetic flux]
 \label{main}
Let $B$ satisfy Assumption \ref{ass:bv} and let $V$ be a non-negative, not identically vanishing function satisfying Assumption \ref{ass:bv}. In addition, assume that $\Phi\in\R\setminus\Z$ and put $\mu:=\big\{|\Phi|\big\}$. Then for all sufficiently small $\alpha>0$ the operator $P-\alpha V$ has exactly $N+1$ negative eigenvalues $\lambda_1(\alpha),\ldots,\lambda_{N+1}(\alpha)$, and as $\alpha\to 0+$ one has for $j=1,\ldots,N$
\begin{equation}
 \label{eq:asympzeromodes}
\lambda_j(\alpha) = - \frac{v_{j- 1}}{m_{j- 1}}\alpha \bigg(1+
\begin{cases}
O\left(\alpha\right) & \text{if}\ j<|\Phi|-1 \\
O\left(\alpha^{\mu}\right) & \text{if}\ j>|\Phi|-1
\end{cases}
\bigg) \,,
\end{equation}
and
$$
\lambda_{N+1}(\alpha) =
- c_\mu v_N^{1/\mu}\alpha^\frac1\mu\big(1+ O(\alpha^{\min\{1, \frac1\mu- 1\}})\big)\,,
$$
where
\begin{equation}\label{c_nu}
c_\mu := \Big(\frac{2^{2\mu- 1}\Gamma(\mu)}{\Gamma(1-\mu)}\Big)^\frac1\mu \,.
\end{equation}
\end{theorem}

The result for integer flux takes the following form.

\begin{theorem}[Case of integer magnetic flux]
 \label{mainint}
Let $B$ satisfy Assumption \ref{ass:bv} and let $V$ be a non-negative, not identically vanishing function satisfying Assumption \ref{ass:bv}. In addition, assume that $\Phi\in\Z$. Then for all sufficiently small $\alpha>0$ the operator $P-\alpha V$ has exactly $N+2$ negative eigenvalues $\lambda_1(\alpha),\ldots,\lambda_{N+2}(\alpha)$.
As $\alpha\to 0+$ the eigenvalues $\lambda_1(\alpha),\ldots,\lambda_{N}(\alpha)$ satisfy \eqref{eq:asympzeromodes} and
$$
\lambda_j(\alpha) = - \frac{v_{j- 1}}{m_{j- 1}}\alpha\Big(1+ O\big(\alpha|\ln\alpha|\big)\Big) \qquad \text{if}\ j=|\Phi|-1 \in\{1,\ldots,N\} \,.
$$
Moreover, if $\Phi\in\Z\setminus\{0\}$, then
\begin{align*}
\lambda_{N+1}(\alpha) & = - 2v_N\frac{\alpha}{|\ln\alpha|} \bigg(1 + O\Big(  \frac{\ln|\ln\alpha |}{|\ln\alpha|}\Big)\bigg) \,,\\
\lambda_{N+2}(\alpha) & = -\exp\Big(- \frac2{v_{N+ 1}\alpha}\big(1+ O(\alpha) \big)\Big)\,,
\end{align*}
and if $\Phi=0$, then
\begin{align*}
\lambda_{1}(\alpha) & = -\exp\Big(- \frac2{v_0^+\alpha}\big(1+ O(\alpha) \big)\Big) \,,\\
\lambda_{2}(\alpha) & = -\exp\Big(- \frac2{v_0^-\alpha}\big(1+ O(\alpha) \big)\Big) \,. 
\end{align*}
\end{theorem}

\bigskip

We emphasize that the meaning of the asymptotics in the case $\Phi\in\Z\setminus\{0\}$ is that $\ln \big|\lambda_{N+2}(\alpha)\big| = - 2v_{N+ 1}^{-1}\alpha^{-1} \big(1+ O(\alpha) \big)$, and similarly in the case $\Phi=0$.

\begin{remark}
The assumption that $V$ is non-negative can be somewhat relaxed. Indeed, our proof shows that for any $k\in\{0,\ldots,N'-1\}$ for which $v_k>0$ (with the obvious modification for $\Phi=0$) the operators $P-\alpha V$ have a negative eigenvalue with the asymptotics given in the theorem. This is in agreement with a result of Weidl \cite{W} who has shown that for $V\not\equiv 0$ one has
$$
\lim_{\alpha\to 0+} N(P-\alpha V) = \#\left\{ k\in\{0,\ldots,N'-1\} :\ v_k \geqslant 0 \right\}
$$
(with the obvious modification for $\Phi=0$). Here $N(P-\alpha V)$ denotes the number of negative eigenvalues, counting multiplicities, of $P-\alpha V$. We have not been able to compute the precise asymptotics of the eigenvalues corresponding to $v_k=0$.
\end{remark}

Let us comment on our assumptions. We assume that $V$ has compact support mainly for the sake of simplicity in order to avoid additional technicalities. With additional work one can probably also replace the support assumption on $B$ by a suitable short-range decay assumption. The assumption that $B$ and $V$ are radial, however, is crucial for us, at least if $|\Phi|\geqslant 1$, since it allows us to avoid orthogonality conditions in the case of several eigenvalues and instead to work with a finite number of \emph{ground state energies}. While physically reasonable we would find it mathematically desirable to remove this assumption.


\subsection{Weak coupling asymptotics for half-line operators}
\label{sec:main1d}

Because of the assumed radial symmetry of $B$ and $V$ Theorems \ref{main} and \ref{mainint} reduce to statements about families of half-line operators. In this subsection we shall formulate a weak coupling result for a more general class of one-dimensional operators. Throughout we work under

\begin{assumption}
\label{ass:w}
\begin{enumerate}
 \item[(a)] The real-valued function $W\in L_{1,\loc}(0,\infty)$ \footnote{$L_{1,\loc}(0,\infty)$ denotes the space of functions which are integrable on all compact subsets of $(0,\infty)$. In view of later applications we insist on the fact that $W$ may have a non-integrable singularity at the origin.}
 satisfies
\begin{equation}
\label{W assumption}
W(r)= \mu^2 r^{-2},\qquad r\geqslant R\,,
\end{equation}
for some $R>0$ and $\mu\geqslant 0$.
\item[(b)] For some $0\leqslant a<1$ and some $M\geqslant 0$ one has for all $\psi\in C_0^\infty(0,\infty)$
\begin{equation}
 \label{eq:formbdd}
\int_0^\infty W_- |\psi|^2 r\,\rmd r \leqslant a \int_0^\infty |\psi'|^2 r\,\rmd r + M \int_0^\infty |\psi|^2 r\,\rmd r \,.
\end{equation}
\item[(c)] For all $\psi\in C_0^\infty(0,\infty)$ one has
\begin{equation}
 \label{eq:nonneg}
\int_0^\infty \left( |\psi'|^2 + W |\psi|^2 \right)\, r\,\rmd r \geqslant 0 \,.
\end{equation}
\end{enumerate}
\end{assumption}

The closure $t$ of the quadratic form on the left hand side of \eqref{eq:nonneg} in the Hilbert space $L_2(\R_+,r\rmd r)$ has domain
$$
\dom t = \big\{ \psi\in L_2(\R_+,r\rmd r) : \ \psi', W_+^{1/2}\psi \in L_2(\R_+,r\rmd r) \big\} \,.
$$
It generates a non-negative self-adjoint operator $T$ which acts on functions $\psi\in\dom T$ according to
$$
T\psi = - r^{-1} (r\psi')' + W\psi \,.
$$
We denote by $\mathcal V$ the space of all functions $V\in L_{1,\loc}(0,\infty)$ with compact support which are form compact with respect to $T$, that is, the operator $(T+1)^{-1/2} V (T+1)^{-1/2}$ is compact. For $V\in\mathcal V$ the operator $T-V$ is defined as usual via its quadratic form with domain $\dom t$. By Weyl's theorem its negative spectrum (if non-empty) consists of discrete eigenvalues of finite multiplicites. We are interested in the behavior of the lowest eigenvalue of the operators $T-\alpha V$ in the weak coupling limit $\alpha\to 0$.

Along with the form domain $\dom t$ of $T$ we shall use the \emph{local form domain} of $T$ defined by
$$
\dom_\loc t := \Big\{ \psi\in L_{2,\loc}\big([0,\infty),r\rmd r\big) : \ \psi', W_+^{1/2}\psi \in L_{2,\loc}\big([0,\infty),r\rmd r\big) \Big\} \,.
$$
We shall say that $\psi$ is a \emph{weak solution} of the equation $T\psi =0$ if $\psi\in\dom_\loc t$ and
$$
t[\phi,\psi]=0
\quad\text{for any}\ \phi\in\dom t
\ \text{with compact support in}\ [0,\infty) \,.
$$
Note that this notion incorporates a boundary condition at zero but no decay condition at infinity.

Here is the key for proving Theorems \ref{main} and \ref{mainint}:

\begin{theorem}[Weak coupling asymptotics for half-line operators]
\label{main1d}
Suppose that Assumption \ref{ass:w} holds and that
\begin{align}\label{psi_0 asymptotic T}
& \text{there exists a positive weak solution $\psi_0$ of the equation $T\psi_0=0$ such that}\notag \\
& \qquad \psi_0(r)= r^{-\mu}, \quad r\geqslant R \,,\\
& \text{with $\mu$ and $R$ from \eqref{W assumption}.} \notag
\end{align}
Let $V\in\mathcal V$ such that
\begin{equation}\label{U assumption T}
v := \int_0^\infty V(r)\psi_0^2(r)r\rmd r> 0 \,.
\end{equation}
Then for all sufficiently small $\alpha> 0$, $T- \alpha V$ has a unique negative eigenvalue $\lambda_\alpha$, and as $\alpha\to 0+$ one has
$$
\lambda_\alpha =
\begin{cases}
-\exp\Big(- \dfrac2{\alpha v} \big(1+ O(\alpha) \big)\Big) \,, & \mu=0 \,,\\
- c_\mu \left(\alpha v\right)^\frac1\mu\big(1+ O(\alpha^{\min\{1, \frac1\mu- 1\}})\big) \,, & \mu\in(0,1) \,, \\
- \dfrac{2\alpha v}{|\ln\alpha|} \bigg(1 + O\Big(  \dfrac{\ln|\ln\alpha|}{|\ln\alpha|}\Big)\bigg) \,, & \mu=1 \,, \\
- \dfrac{\alpha v}{\int_0^\infty\psi_0^2(r)rdr} \bigg(1+ 
\begin{cases} O(\alpha^{\min\{1, \mu- 1\}})\,,& \mu\neq 2\\ O( \alpha|\ln\alpha|) \,,& \mu= 2\end{cases}
\bigg) \,, & \mu>1\,,
\end{cases}
$$
where $c_\mu$ is given by \eqref{c_nu}.
\end{theorem}

We remark that for $\mu>1$ the function $\psi_0$ is square-integrable, and hence $0$ is an eigenvalue of $T$. In this case, the leading order asymptotics of $\lambda_\alpha$ follow from the abstract arguments of \cite{S2}. Since in our approach there is hardly any difference between the cases $\mu>1$ and $0\leqslant\mu\leqslant 1$, we include an independent proof which, moreover, yields a remainder estimate. We note in passing that by concavity the $O$-term is positive for $\mu>1$.

In general, the coefficient $v$ in \eqref{U assumption T} depends implicitly on the background potential $W$ through the function $\psi_0$. When applying Theorem \ref{main1d} in the proof of Theorems \ref{main} and \ref{mainint}, however, $W$ will have a specific form which allows to determine $\psi_0$ explicitly. Hence for given $V$ the coefficient $v$ can be computed explicitly.

\bigskip

The main point in Theorem \ref{main1d} is that it connects the existence of a positive solution of the equation $T\psi_0 =0$ with a certain behavior at infinity to the spectral information about the existence of negative eigenvalues. This connection is \emph{quantitative} in the sense that the decay of $\psi_0$ determines the size of the eigenvalue in the weak coupling limit. An initial step in our proof of Theorem \ref{main1d} will be a \emph{qualitative} version of this correspondence, which we single out as Proposition \ref{vl} below. We emphasize that this qualitative statement is a well-studied feature of second order elliptic operators \cite{P}, though our assumptions on the potentials seem to be much weaker than those typically imposed in the literature.

\begin{proposition}[Characterization of criticality]
\label{vl}
Suppose that Assumption \ref{ass:w} holds. Then Assumption \eqref{psi_0 asymptotic T} is equivalent to each of the following:
\begin{enumerate}
\item[\emph{(i)}] 
For \emph{any} non-negative $0\not\equiv V\in\mathcal V$ the spectrum of $T- V$ in $(-\infty, 0)$ is nonempty.
\item [\emph{(ii)}]
For \emph{some} $V\in\mathcal V$ and any $\alpha>0$ the spectrum of $T- \alpha V$ in $(-\infty, 0)$ is nonempty.
\end{enumerate}
\end{proposition}

The structure of the proofs of Theorem \ref{main1d} and Proposition \ref{vl} is as follows. In Section \ref{sec:upper} we assume the existence of a function $\psi_0$ as in \eqref{psi_0 asymptotic T}. Starting from this function we construct a family of trial functions with negative energy. This implies part (i) in Proposition \ref{vl} and gives us the upper bound claimed in Theorem \ref{main1d}. In Section \ref{sec:lower} we assume the existence of a potential $V$ as in part (ii) of Proposition \ref{vl}. Starting from the family of eigenfunctions $\psi_\alpha$ of the corresponding operators we construct a function $\psi_0$ as in \eqref{psi_0 asymptotic T}. Moreover, controlling the convergence of $\psi_\alpha$ to $\psi_0$ will allow us to prove the lower bound claimed in Theorem \ref{main1d}. Uniqueness of the eigenvalue is the content of Lemma \ref{simple}.


\section{Proof of Theorem \ref{main1d} and Proposition \ref{vl}}\label{sec:proofs}

\subsection{The upper bound}\label{sec:upper}

In this subsection we prove that \eqref{psi_0 asymptotic T} implies (i) in Proposition \ref{vl} and we derive the upper bound on $\lambda_\alpha$ stated in Theorem \ref{main1d}. Both follow immediately from

\begin{proposition}\label{upper}
Assume that there exists a function $\psi_0$ as in \eqref{psi_0 asymptotic T} and let $V\in\mathcal V$ satisfy \eqref{U assumption T}. Then $\lambda_\alpha := \inf\spec (T-\alpha V)<0$ for any $\alpha>0$ and, as $\alpha\to 0$,
$$
\lambda_\alpha \leqslant
\begin{cases}
-\exp\Big(- \dfrac2{\alpha v} \left(1+ \const \alpha\right)\Big) \,, & \mu=0 \,,\\
- c_\mu \left(\alpha v\right)^\frac1\mu(1- \const\alpha^{\min\{1, \frac1\mu- 1\}}) \,, & \mu\in(0,1) \,, \\
- \dfrac{2\alpha v}{|\ln\alpha|} \Big(1 - \const  \dfrac{\ln|\ln\alpha |}{|\ln\alpha |}\Big) \,, & \mu=1 \,, \\
- \dfrac{\alpha v}{\int_0^\infty\psi_0^2(r)rdr} \,, & \mu>1\,,
\end{cases}
$$
where $c_\mu$ is given by \eqref{c_nu}.
\end{proposition}

\begin{proof}{}
First assume that $\mu> 1$. Then $\psi_0$ is square-integrable and $t[\psi_0]= 0$. Hence by the variational principle
$$
\lambda_\alpha \leqslant \frac{t[\psi_0] - \alpha \int_0^\infty V(r)\psi_0^2(r)r\,\rmd r}{\int_0^\infty\psi_0^2(r)r\,\rmd r}
= - \Big(\int_0^\infty\psi_0^2(r)r\,\rmd r\Big)^{-1} \alpha v \,,
$$
as claimed. In the remainder of this proof we shall assume that $0\leqslant\mu\leqslant 1$. Let $R>0$ such that $W(r)=\mu^2 r^{-2}$ and $V(r)=0$ for $r\geqslant R/2$ and define for any $\kappa> 0$
\begin{equation}\label{eq:trial}
\varphi_\kappa(r):= 
\begin{cases}
\psi_0(r)\,, & r\leqslant R\,,\\ 
\dfrac{K_\mu\big(\kappa r\big)}{R^\mu K_\mu(\kappa R)}\,,& r> R\,.
\end{cases}
\end{equation}
Here $K_\mu$ is the modified Bessel function of order $\mu$, see \cite{GR}.
The function $\phi_\kappa$ belongs to the form domain of $T-\alpha V$ and we claim that for small $\kappa>0$
\begin{equation}\label{numerator asymptotic}
t[\varphi_\kappa]= 
\begin{cases}
-(\ln\kappa)^{-1}+ O\big(|\ln\kappa|^{-2}\big)\,,& \mu= 0\,,\\ 
2^{1- 2\mu}(1- \mu)\dfrac{\Gamma(1- \mu)}{\Gamma(\mu)}\kappa^{2\mu} + O(\kappa^{2\min\{1, 2\mu\}})\,,& \mu\in (0, 1)\,, \\
O(\kappa^2) \,,& \mu= 1\,,
\end{cases}
\end{equation}
and
\begin{equation}\label{denominator asymptotic}
\|\varphi_\kappa\|^2= 
\begin{cases}
O\big(\kappa^{-2}|\ln\kappa|^{-2}\big)\,,& \mu= 0\,,\\ 
2^{1- 2\mu}\mu\dfrac{\Gamma(1- \mu)}{\Gamma(\mu)}\kappa^{-2(1-\mu)}+ O(\kappa^{2\min\{-1+ 2\mu, 0\}})\,,& \mu\in (0, 1) \,,\\
-\ln\kappa+ O(1)\,,& \mu= 1\,.
\end{cases}
\end{equation}
To prove the first relation, we multiply the equation $T\psi_0=0$ by $r\psi_0$ and integrate by parts. Using the definition of a weak solution and that $\psi_0\in\dom_\loc t$ one sees that there appear no boundary terms at zero and one obtains
\begin{equation}\label{eq:minlocal}
\int_0^R \left(|\psi_0'|^2 + W|\psi_0|^2\right)\,r\rmd r = - \mu R^{-2\mu} \,.
\end{equation}
Hence
\begin{equation}\label{t[varphi_lambda]}
t[\varphi_\kappa] = -\mu R^{-2\mu} + R^{-2\mu} K_\mu^{-2}(\kappa R) \,
A(\kappa)
\end{equation}
with
$$
A(\kappa) 
:= \int_R^\infty \Big( \kappa^2 \big( K_\mu'(\kappa r)\big)^2+ \frac{\mu^2}{r^2}K_\mu^2(\kappa r)\Big) r\,\rmd r \,.
$$
From \cite[5.54.2, 8.486.13]{GR} we get
\begin{equation*}\label{I_lambda}
\begin{split}
A(\kappa) 
&= \int_R^\infty\big(\kappa^2 rK_{\mu+ 1}^2(\kappa r)- 2\kappa\mu K_{\mu+ 1}(\kappa r)K_\mu(\kappa r)+ 2\frac{\mu^2}{r}K_\mu^2(\kappa r)\big)\rmd r\\ 
&= \int_R^\infty \kappa rK_{\mu+ 1}^2(\kappa r)\rmd(\kappa r)+ 2\mu\int_R^\infty K_\mu(\kappa r)\rmd K_\mu(\kappa r)\\&= \frac{\kappa^2 R^2}2\big(K_\mu(\kappa R)K_{\mu+ 2}(\kappa R)- K_{\mu+ 1}^2(\kappa R)\big)- \mu K_\mu^2(\kappa R) \,,
\end{split}
\end{equation*}
and \eqref{numerator asymptotic} follows from the asymptotics of the modified Bessel functions for small arguments \cite[8.485, 8.445, 8.446]{GR}. To prove \eqref{denominator asymptotic} we write
\begin{equation*}
\|\varphi_\kappa\|^2= \int_0^R |\psi_0|^2r\rmd r + R^{-2\mu} K_\mu^{-2}(\kappa R)\, J(\kappa)
\end{equation*}
where 
\begin{align*}
J(\kappa)&:= \int_R^\infty K_\mu^2(\kappa r) r\,\rmd r= \frac{R^2}{2}\big(K_{1- \mu}(\kappa R)K_{1+ \mu}(\kappa R)- K_\mu^2(\kappa R)\big) \,.
\end{align*}
In the last identity we used again \cite[5.54.2]{GR}, and asymptotics \eqref{denominator asymptotic} follow as before from the asymptotics of the modified Bessel functions.

Relation \eqref{numerator asymptotic} shows that for any $\alpha$ we can choose a sufficiently small $\kappa$ to make the quotient $(t[\varphi_\kappa]-\alpha v)/ \|\varphi_\kappa\|^2$ negative. To obtain an upper bound on $\lambda_\alpha$ we minimize this quotient to leading order in $\alpha$. Namely, we choose
\begin{align*}
 \kappa^2 = 
\begin{cases}
C\exp\left(-2/(\alpha v) \right)\,, & \mu=0 \,, \\
 \Big(\dfrac{\Gamma(\mu) \alpha v}{2^{1- 2\mu}\Gamma(1- \mu)}\Big)^\frac1\mu \,, & \mu\in (0,1)\,, \\
\dfrac{2\alpha v}{- \ln \alpha }\,, & \mu=1 \,,
\end{cases}
\end{align*}
where $C$ is a sufficiently small constant in case $\mu=0$. The claimed upper bounds easily follow using \eqref{numerator asymptotic} and \eqref{denominator asymptotic}.
\end{proof}

\begin{remark}
The definition of $\varphi_\kappa$ also makes sense for $\mu>1$. For later reference we note that in this case \eqref{numerator asymptotic} and \eqref{denominator asymptotic} take the form
\begin{equation}\label{numerator asymptotic2}
t[\varphi_{\kappa}]= 
\begin{cases}
O\left(\kappa^{2 \min\{\mu, 2\}} \right)\,, & \mu\neq 2\,,\\ 
O\left(\kappa^4 |\ln\kappa| \right)\,, & \mu= 2\,,
\end{cases}
\end{equation}
and
\begin{equation}\label{denominator asymptotic2}
\|\varphi_\kappa\|^2 
= \|\psi_0\|^2 + 
\begin{cases}
O\left(\kappa^{2\min\{\mu- 1, 1\}} \right)\,, & \mu\neq 2\,,\\ 
O\left(\kappa^2 |\ln\kappa| \right)\,, & \mu= 2\,.
\end{cases}
\end{equation}
This is proved similarly as in the case $0\leqslant\mu\leqslant 1$.
\end{remark}


\subsection{The lower bound}\label{sec:lower}

Our goal in this section is to prove that (ii) in Proposition \ref{vl} implies the existence of a function $\psi_0$ as in \eqref{psi_0 asymptotic T} and to prove the lower bound stated in Theorem \ref{main1d}.

Throughout we assume that $T$ is non-negative and we fix $V\in \mathcal V$ such that the negative spectrum of $T-\alpha V$ is non-empty for any $\alpha>0$. By Weyl's theorem, $\lambda_\alpha := \inf\spec (T-\alpha V)<0$ is an eigenvalue and we normalize the corresponding eigenfunction $\psi_\alpha$ by
\begin{equation}\label{eq:psinorm}
\psi_\alpha(R)= R^{-\mu} \,.
\end{equation}
Here $R> 0$ is chosen such that $W(r)= \mu^2r^{-2}$ and $V(r)=0$ for $r\geqslant R/2$.


\subsubsection{Existence of a virtual ground state}

We prove that $\psi_\alpha$, normalized by \eqref{eq:psinorm}, have a limit as $\alpha\to +0$ and that this limit is a weak solution of the equation $T\psi=0$. More precisely, we have

\begin{lemma}\label{vgs}
 The functions $\psi_\alpha$ converge pointwise and in $L_{2,\loc}([0,\infty),r\rmd r)$ to a function $\psi_0$ satisfying the properties stated in \eqref{psi_0 asymptotic T}. Moreover,
\begin{equation}\label{eq:conv}
0 \leqslant \int_0^R \left(|(\psi_\alpha-\psi_0)'|^2 + W|\psi_\alpha-\psi_0|^2\right)\,r\rmd r \leqslant \const \alpha^2
\end{equation}
for all sufficiently small $\alpha>0$, and
\begin{align}\label{norm on compact}
\int_0^R|\psi_\alpha|^2\, r\rmd r & = \int_0^R|\psi_0|^2\, r\rmd r + O(\alpha)\,, \\
\label{U integral}
\int_0^R V |\psi_\alpha|^2\, r\rmd r & = \int_0^R V |\psi_0\big|^2\, r\rmd r + O(\alpha) \,.
\end{align}
\end{lemma}

For the proof of this lemma we need a rough a priori lower bound on the lowest eigenvalue, which we single out as

\begin{lemma}\label{linear}
For all sufficiently small $\alpha>0$ one has
\begin{equation}\label{eq:linear}
|\lambda_\alpha| \leqslant \const \alpha\,.
\end{equation}
\end{lemma}

\begin{proof}{}
First note that $\lambda_\alpha\to 0$ as $\alpha\to 0$. Indeed, this follows from the fact that the operator $(T+\epsilon)^{-1/2} V (T+\epsilon)^{-1/2}$ is bounded for any $\epsilon>0$, which in turn follows from the form boundedness of $V$. Now by the variational principle, $\lambda_\alpha$ is the infimum over linear functions of $\alpha$, and hence $\lambda_\alpha$ is a concave function of $\alpha$. Hence $\lambda_\alpha \geqslant \alpha \lambda_1$ for $0<\alpha\leqslant 1$, proving \eqref{eq:linear}.
\end{proof}

\begin{proof}[Proof of Lemma \ref{vgs}]
 \emph{First step. Definition of $\psi_0$.}
Let $T_R$ be the self--adjoint operator in $L_2(0,R,r\rmd r)$ associated with the closure of the quadratic form
$$
t_R[\psi] := \int_0^R \left(|\psi'|^2 + W|\psi|^2\right)\,r\,\rmd r\,,
\quad \psi\in C_0^\infty(0,R) \,.
$$
Using \eqref{eq:formbdd} one easily shows that $T_R$ has discrete spectrum. By \eqref{eq:nonneg} $T_R$ is non-negative. We claim that $T_R$ is actually positive definite. Indeed, otherwise $T_R$ would have a zero eigenvalue with corresponding eigenfunction $\tilde\psi_0$. The extension of $\tilde\psi_0$ by zero belongs to the form domain $\dom t$ and according to \eqref{eq:nonneg} minimizes the quadratic form $t$. Hence it belongs even to the operator domain $\dom T$. But this would imply that $\tilde\psi_0'(R)=0$ and therefore by the unique solvability of the Cauchy problem $\tilde\psi_0\equiv 0$, a contradiction.

Fix a function $F\in\dom T$ with $F(R)=R^{-\mu}$ and denote $f:=TF$. We define on $(0, R)$
$$
\psi_0 := F - T_R^{-1} f \,.
$$
(Strictly speaking, $T_R^{-1}$ is applied to the restriction of $f$ to the interval $(0,R)$.) Note that $\psi_0$ satisfies
\begin{equation}\label{psi_0 equation}
- r^{-1}(r \psi_0')' + W \psi_0 = 0
\qquad \text{in } (0,R)\,,
\qquad \psi_0(R)=R^{-\mu}\,,
\end{equation}
and belongs to $\dom_\loc t_R$ (which is defined similarly to $\dom_\loc t$). Moreover we claim that,
\begin{equation}\label{eq:resenergy}
\int_0^R \left(|\psi_0'|^2 +W|\psi_0|^2\right)\,r\rmd r
\qquad\text{is finite}.
\end{equation}
Indeed,
\begin{align*}
& \int_0^R \left(|\psi_0'|^2 +W|\psi_0|^2\right)\,r\rmd r \\
& \quad = t_R[\psi_0-F] - \int_0^R \left(|F'|^2 + W|F|^2\right) r\,\rmd r  
+ 2\re \int_0^R \left(\overline{F'}\psi_0' + W\overline F \psi_0 \right) r\,\rmd r \\
& \quad = t_R[\psi_0-F] - \int_0^R \overline f F r\,\rmd r - R \overline{F'(R)}F(R) 
+ 2\re \int_0^R \overline{f}\psi_0 r\,\rmd r \\
& \quad\quad\quad + 2R \re \overline{F'(R)}\psi_0(R) \,,
\end{align*}
where all the terms on the right hand side are finite. (Indeed, one easily checks that $F'(R)$ is well defined for $F\in\dom T$.) Note that when integrating by parts no boundary terms at zero appear since $F$ and $\psi_0$ belong to $\dom_\loc t_R$.

\emph{Second step. Proof of \eqref{eq:conv}.}
Let $0\leqslant\chi\leqslant 1$ be a smooth function on $[0,R]$ such that $\chi\equiv 1$ on $\big[0,R/2\big]$ and $\chi\equiv 0$ on $\big[3R/4, R\big]$. We first claim that
\begin{equation}
\label{eq:phialpha}
\psi_\alpha-\psi_0= \varphi_\alpha := T_R^{-1/2}\big(\alpha(T_R^{-1/2} V T_R^{-1/2}) T_R^{1/2}\chi\psi_\alpha + \lambda_\alpha  T_R^{1/2}\psi_\alpha\big).
\end{equation}
Note that this identity shows that the definition of $\psi_0$ is independent of the choice of $F$. On the other hand, the above formula is also independent of the choice of $\chi$.

Since $T_R^{-1/2} V T_R^{-1/2}$ is a bounded operator and $\chi\psi_\alpha$ belongs to the form domain of $T_R$, the function $\varphi_\alpha$ is well-defined and belongs to the form domain of $T_R$. For any $\phi\in\dom t_R$ one has (with $\tilde\phi$ denoting its extension by zero)
\begin{align*}
 t_R[\phi,\phi_\alpha] & = \int_0^R (\alpha V +\lambda_\alpha) \overline\phi \psi_\alpha r\,\rmd r
= t[\tilde\phi,\psi_\alpha] = t_R[\phi,\psi_\alpha-F] + (\tilde\phi,f) \\ 
& = t_R[\phi,\psi_\alpha-F] + t_R[\phi,F-\psi_0]
= t_R[\phi,\psi_\alpha-\psi_0] \,.
\end{align*}
Since $T_R$ is positive definite this implies $\varphi_\alpha = \psi_\alpha-\psi_0$, proving \eqref{eq:phialpha}.

We denote by $\|\cdot\|_R$ the norm in $L_2(0,R,rdr)$. It follows from \eqref{eq:phialpha}, \eqref{eq:linear}, the form-boundedness of $V$ and the positive definiteness of $T_R$ that
\begin{equation}\label{eq:conv1}
 t_R[\psi_\alpha-\psi_0]^{1/2} = t_R[\varphi_\alpha]^{1/2} \leqslant \const \alpha \big( \|\psi_\alpha\|_R + t_R[\chi\psi_\alpha]^{1/2} \big) \,.
\end{equation}
To estimate the first term on the right hand side of \eqref{eq:conv1} we use that $T_R$ is positive definite and hence
\begin{equation}\label{eq:conv2}
\|\psi_\alpha\|_R \leqslant  \|\psi_0\|_R + \|\varphi_\alpha\|_R 
\leqslant \const \big( 1+ t_R[\varphi_\alpha]^{1/2} \big) \,.
\end{equation}
To estimate the second term we use that according to \eqref{eq:resenergy} (recall that $W(r)=\mu^2 r^{-2}$ if $\chi(r)<1$)
\begin{align}\label{eq:conv3}
t_R[\chi\psi_\alpha] 
& \leqslant 2 \big( t_R[\chi\psi_0] +  t_R[\chi\varphi_\alpha] \big) \notag \\
& \leqslant \const \Big( \int_0^R \left( |\psi_0'|^2 + W|\psi_0|^2\right)\,r\rmd r + \|\psi_0\|_R^2 + t_R[\varphi_\alpha] + \|\varphi_\alpha\|_R^2\Big) \notag \\
& \leqslant \const \big( 1 + t_R[\varphi_\alpha] \big) \,.
\end{align}
Combining \eqref{eq:conv1}, \eqref{eq:conv2} and \eqref{eq:conv3} we obtain \eqref{eq:conv}.

Note that again by positive definiteness, \eqref{eq:conv} implies that
\begin{equation}\label{eq:convnorm}
\int_0^R |\psi_\alpha-\psi_0|^2\,r\rmd r \leqslant \const\alpha^2\,,
\end{equation}
which implies \eqref{norm on compact}. Relation \eqref{U integral} follows from \eqref{eq:conv} by writing
$$
\int_0^R V |\psi_\alpha|^2\, r\rmd r - \int_0^R V |\psi_0\big|^2\, r\rmd r 
= \int_0^R V |\varphi_\alpha|^2\, r\rmd r + 2 \int_0^R V \varphi_\alpha \chi\psi_0 \, r\rmd r
$$
and using that $ T_R^{-1/2} V T_R^{-1/2}$ is bounded and that $\chi\psi_0$ belongs to the form domain of $T_R$.

\emph{Third step. Pointwise convergence.}
By explicit solution,
$$
\psi_\alpha(r)= \frac{K_{\mu}\big(\sqrt{|\lambda_\alpha|}\,r\big)}{R^\mu\ K_{\mu}\big(\sqrt{|\lambda_\alpha|}\,R\big)}\,,
\qquad r\in [R/2, \infty) \,.
$$
Recall that $\psi_0$ satisfies \eqref{psi_0 equation} and hence belongs to the two-dimensional space 
of solutions of $-\psi_0''+\mu^2 r^{-2}\psi_0=0$ on $[R/2,R]$. From convergence \eqref{eq:convnorm} and the asymptotics of the Bessel functions we conclude that $\psi_0(r)= r^{-\mu}$ on $\big[R/2, R\big]$ and we can extend $\psi_0$ to $(R,\infty)$.

Recalling \eqref{psi_0 equation} and \eqref{eq:linear} we see that the coefficients in the differential equation satisfied by $\psi_\alpha$ converge in $L_{1,\loc}$ as $\alpha\to0$ to those of the equation satisfied by $\psi_0$. Moreover, for any fixed $r\geqslant R/2$ one has $\psi_\alpha(r)\to\psi_0(r)$ and $\psi_\alpha'(r)\to\psi_0'(r)$. By standard ODE results (see, e.g., \cite[Thm. 2.1]{Wm}) this implies that $\psi_\alpha$ converge pointwise to $\psi_0$ on $(0,r)$.

Finally, $\psi_\alpha$ are the eigenfunctions corresponding to the lowest eigenvalue and therefore non-negative. Hence their pointwise limit $\psi_0$ is so as well. By standard ODE results (an elementary Harnack's inequality), $\psi_0$ is positive on $(0,\infty)$.
\end{proof}

\begin{lemma}\label{simple}
For all sufficiently small $\alpha>0$ the operator $T-\alpha V$ has only one eigenvalue.
\end{lemma}

\begin{proof}{}
Imposing a Dirichlet boundary condition is a rank one perturbation of the resolvent and can create at most one negative eigenvalue. Hence it suffices to prove that the operator $T-\alpha V$ with an additional Dirichlet boundary condition at $R$ is non-negative. This is obvious for the part on $(R,\infty)$. The part on $(0,R)$ coincides with the operator $T_R-\alpha V$ from the previous proof. Since $T_R$ is positive definite and $T_R^{-1/2} V T_R^{-1/2}$ is bounded, we obtain the claim.
\end{proof}


\subsubsection{The lower bound}

In the proof of the upper bound we have used identity \eqref{eq:minlocal} for the virtual ground state $\psi_0$. For the proof of the lower bound we need a corresponding inequality for all $\psi$. This is the content of

\begin{lemma}\label{energycomp}
For any $\psi\in\dom t$
$$
\int_0^R \big(|\psi'|^2 + W|\psi|^2\big)\,r\,\rmd r \geqslant -\mu \big|\psi(R)\big|^2 .
$$
\end{lemma}

Here we use $R$ as defined at the beginning of this section, but any $R$ with $W(r)=(\mu/r)^2$ for $r\geqslant R$ would do. Note also that the value $\psi(R)$ is well-defined by the embedding theorem.

\begin{proof}{}
If $\psi(R)=0$, the assertion follows since $T\geqslant 0$. Hence by homogeneity, we may assume that $\psi(R)= R^{-\mu}$. For $\varepsilon>0$ we define $\psi_\epsilon(r):=\psi(r)$ for $0\leqslant r\leqslant R$ and
\begin{equation*}
\psi_\varepsilon(r) := r^{-\mu} \rme^{-\varepsilon (r- R)}, \quad r\geqslant R.
\end{equation*}
Since $T$ is non-negative and $\psi_\epsilon\in\dom t$, we get
\begin{equation*}
\int_0^R \big(|\psi'|^2 + W|\psi|^2\big)\,r\,\rmd r
\geqslant - \int_R^\infty \big(|\psi_\varepsilon'|^2 + \mu^2
r^{-2} |\psi_\varepsilon|^2\big)\,r\,\rmd r =: - I_\varepsilon \,.
\end{equation*}
Elementary calculations show that $I_\varepsilon \to \mu R^{-2\mu}$ as $\varepsilon\to 0$, which proves the assertion.
\end{proof}

Now we are ready to complete the proof of Theorem \ref{main1d}.

\begin{proposition}\label{lower}
Let $V\in\mathcal V$ satisfy \eqref{U assumption T}. Then $\lambda_\alpha := \inf\spec (T-\alpha V)$ satisfies
$$
\lambda_\alpha \geqslant
\begin{cases}
-\exp\Big(- \dfrac2{\alpha v} \left(1- \const \alpha\right)\Big) \,, & \mu=0 \,,\\
- c_\mu \left(\alpha v\right)^\frac1\mu(1+ \const\alpha^{\min\{1, \frac1\mu- 1\}}) \,, & \mu\in(0,1) \,, \\
- \dfrac{2\alpha v}{|\ln\alpha|} \Big(1 + \const  \dfrac{\ln|\ln\alpha |}{|\ln\alpha |}\Big) \,, & \mu=1 \,, \\
- \dfrac{\alpha v}{\int_0^\infty\psi_0^2(r)r\rmd r} \bigg(1+ \const 
\begin{cases}\alpha^{\min\{1, \mu- 1\}},& \mu\neq 2\\ \alpha|\ln\alpha|\,,& \mu= 2\end{cases}
\bigg) \,, & \mu>1\,,
\end{cases}
$$
where $c_\mu$ is given by \eqref{c_nu}.
\end{proposition}

\begin{proof}[Proof of Proposition \ref{lower}]
As before, let $\psi_\alpha$ be the eigenfunctions corresponding to $\lambda_\alpha$ normalized by \eqref{eq:psinorm} and let $\varphi_\kappa$ be the functions defined in \eqref{eq:trial}. Note that 
\begin{equation}\label{explicit solution}
\psi_\alpha(r)= \varphi_{\kappa_\alpha}(r), \quad r\geqslant R/2 \,,
\end{equation}
where $\lambda_\alpha = -\kappa_\alpha^2$. In order to find a lower bound on $\lambda_\alpha$ we write
\begin{equation}\label{lambda_alpha}
\lambda_\alpha \|\psi_\alpha\|^2 = t[\psi_\alpha]- \alpha\int_0^\infty V|\psi_\alpha|^2r\rmd r \,.
\end{equation}
Using Lemma \ref{energycomp}, \eqref{eq:minlocal} and \eqref{U integral} the right hand side can be estimated from below according to
\begin{equation}\label{numerator below}
t[\psi_\alpha]- \alpha\int_0^\infty V|\psi_\alpha|^2r\rmd r
\geqslant t[\varphi_{\kappa_\alpha}]- \alpha\int_0^\infty V\psi_0^2r\rmd r- \const\alpha^2 \,,
\end{equation}
and in order to estimate the left hand side we use \eqref{norm on compact} and obtain
\begin{equation}\label{denominator below}
\|\psi_\alpha\|^2\geqslant \|\varphi_{\kappa_\alpha}\|^2- \const\alpha.
\end{equation}
Plugging \eqref{numerator below} and \eqref{denominator below} into \eqref{lambda_alpha} and using the a-priori bound \eqref{eq:linear} yields
\begin{equation}\label{eq:prooflower}
t[\varphi_{\kappa_\alpha}] - \lambda_\alpha \|\varphi_{\kappa_\alpha} \|^2
\leqslant \alpha v \, (1 + \const \alpha ) \,.
\end{equation}
The assertion will now be an easy consequence of the behavior of $t[\varphi_{\kappa_\alpha}]$ and $\|\varphi_{\kappa_\alpha}\|^2$ which was established in the previous section.

Indeed, assume first that $0<\mu<1$. Using \eqref{numerator asymptotic}, \eqref{denominator asymptotic} and the fact that $\alpha\leqslant \const |\lambda_\alpha|^\mu$, which follows from Proposition \ref{upper}, we deduce from \eqref{eq:prooflower} that
$$
2^{1-2\mu} \frac{\Gamma(1-\mu)}{\Gamma(\mu)} |\lambda_\alpha|^\mu \big(1-\const |\lambda_\alpha|^{\min\{1-\mu,\mu\}}\big)
\leqslant \alpha v \,.
$$
Together with the a-priori fact that $\lambda_\alpha\to 0$, established in Lemma \ref{linear}, one easily obtains the assertion in the case $0<\mu<1$.

In the cases $\mu=0$ and $\mu\geqslant 1$ we proceed similarly and we only sketch the necessary changes. 
In order to remove the $\alpha^2$-term from the right hand side of \eqref{eq:prooflower} we use the rough bounds $\alpha \leqslant \const \big|\ln |\lambda_\alpha| \big|^{-1} $ if $\mu=0$, $\alpha \leqslant \const |\lambda_\alpha| \big|\ln |\lambda_\alpha|\big|$ if $\mu=1$ and $\alpha\leqslant \const |\lambda_\alpha|$ if $\mu>1$, which are deduced from Proposition \ref{upper}. Moreover, if $\mu=0$ we estimate $\|\varphi_{\kappa_\alpha} \|^2\geqslant0$ in \eqref{eq:prooflower} and if $\mu\geqslant 1$ we estimate $t[\varphi_{\kappa_\alpha}] \geqslant 0$. We use asymptotics \eqref{numerator asymptotic} for $\mu=0$, \eqref{denominator asymptotic}
for $\mu=1$ and \eqref{denominator asymptotic2} for $\mu>1$. Finally, in case $\mu=1$ to pass from the bound $\big|\lambda_\alpha \ln |\lambda_\alpha|\big| \leqslant 2\alpha v \big(1+\const |\ln\alpha|^{-1}\big)$ to the claimed lower bound we use the following Lemma \ref{invlog}.
\end{proof}

In the previous proof we used

\begin{lemma}\label{invlog}
Let $f(t)=-t\ln t$ for $0<t\leqslant \rme^{-1}$ and $f^{-1}$ the inverse function. Then $f^{-1}$ is increasing and
$$
f^{-1}(s) = \frac{-s}{\ln s} \left(1+ O\Big(\frac{\ln |\ln s|}{|\ln s|}\Big)\right)
\qquad\text{as} \ s\to 0+ \,.
$$
\end{lemma}

The proof of this lemma is elementary and will be omitted.


\subsection{Absence of negative eigenvalues}

In this subsection we briefly comment on the case where Assumption \ref{ass:w} is satisfied but the equivalent conditions in Proposition \ref{vl} fail. We will need this in the proof of Theorems \ref{main} and \ref{mainint}. Note that if the weak solution of $T\psi_0=0$ is different from $r^{-\mu}$ on $[R,\infty)$, then it has to increase like $\ln r$ if $\mu=0$ or like $r^\mu$ if $\mu>0$.

\begin{proposition}\label{hardy}
Assume that there exists a positive weak solution $\psi_0$ of the equation $T\psi_0=0$ such that
the limit $\lim_{r\to\infty} (\ln r)^{-1} \psi_0(r)$ if $\mu=0$ and $\lim_{r\to\infty} r^{-\mu} \psi_0(r)$ if $\mu>0$ exists and is non-zero. Then for all $\psi\in C_0^\infty(0,\infty)$ and for all non-negative measurable $V$ one has
$$
t[\psi] \geqslant \frac14 \Big( \sup_{0<r<\infty} \int_0^r V \psi_0^2 \rho\,\rmd\rho \int_r^\infty \psi_0^{-2} \rho^{-1}\,\rmd\rho \Big)^{-1}  \int_0^\infty V |\psi|^2 r\,\rmd r \,.
$$
\end{proposition}

Note that the integral $\int_r^\infty \psi_0^{-2} \rho^{-1}\,\rmd\rho$ is finite for any $r>0$ because of the assumed growth of $\psi_0$. The above supremum, however, may be finite or infinite depending on $V$. In the latter case one actually can show, arguing as below, that there is \emph{no} $c>0$ such that $t[\psi] \geqslant c \int_0^\infty V |\psi|^2 r\,\rmd r$ for all $\psi$.

\begin{proof}{}
Writing $\psi = \psi_0 \phi$ and using the equation for $\psi_0$ we find that
$$
t[\psi] = \int_0^\infty |\phi'|^2 \psi_0^2 \,r\,\rmd r \,,
\qquad
 \int_0^\infty V |\psi|^2 r\,\rmd r = \int_0^\infty V |\phi|^2 \psi_0^2 \,r\,\rmd r \,.
$$
The assertion now follows from a classical theorem by Muckenhoupt \cite[Thm. 1.3.1/3]{M}.
\end{proof}


\section{Proof of Theorems \ref{main} and \ref{mainint}}\label{sec:proofmain}

\subsection{Angular momentum decomposition}
\label{sec:symm}

In this subsection we exploit the fact that $B$ and $V$ are radially symmetric, so that $P-\alpha V$ can be decomposed according to the eigenvalues of the angular momentum operator. We introduce polar coordinates $(r,\varphi)$ in $\R^2$ and write
\begin{equation}
 \label{eq:decomp}
\psi(x) = \frac1{\sqrt{2\pi}} \sum_{m\in\Z} \psi_m(r) \rme^{im\phi} \,.
\end{equation}
This establishes a unitary equivalence between $\psi\in L_2(\R^2,\C^2)$ and sequences 
\begin{equation*}
(\psi_m)\in\sum_m\oplus L_2(\R_+,r\rmd r,\C^2).
\end{equation*}
By gauge invariance we may choose the magnetic vector potential $A$ as
\begin{equation}\label{vector potential}
A(r, \varphi):= b(r)(-\sin\varphi, \cos\varphi)\,, \quad b(r):=\frac{1}{r}\int_0^r B(\rho)\,\rho\, \rmd\rho \,.
\end{equation}
Note that $A\in L_{2,\loc}(\R^2,\R^2)$ in view of \eqref{eq:assbv}. Plugging decomposition \eqref{eq:decomp} into the quadratic form of the Pauli operator we find that
$$
\int_{\R^2} \big| \mathbf{\sigma}\cdot(-i\nabla+A)\psi\big|^2 \,\rmd x
= \sum_m \big( t_m^+[\psi_m^+] + t_m^-[\psi_m^-] \big) \,,
$$
where
$$
t_m^\pm[f] := \int_0^\infty \big(|f'|^2 + W_m^\pm |f|^2\big) \,r\,\rmd r
$$
and
$$
W_m^\pm(r) : = \big(b(r)+ mr^{-1}\big)^2 \pm B(r) \,.
$$
We denote by $T^\pm_m$ the self-adjoint operator in $L_2(\R_+,rdr)$ corresponding to the form $t_m^\pm$. We note that $C_0^\infty(0,\infty)$ is a form core for this operator. Indeed, by the arguments of \cite{S3} $C_0^\infty(\R^2\setminus\{0\})$ is a form core of $(D+A)^2$ and since $B$ is relatively form compact (see the following subsection), the same is true for $P$.

We introduce the functions
$$
\omega_m^+(r) :=
\begin{cases}
r^{m} \rme^{-\xi(r)} & \text{if}\ m\geqslant 0\,,\\
r^{m} \rme^{-\xi(r)}\int_0^r \rme^{2\xi(\rho)}\rho^{-2m-1} \,\rmd\rho & \text{if}\ m\leqslant -1\,,
\end{cases}
$$
and
$$
\omega_m^-(r) :=
\begin{cases}
r^{-m} \rme^{\xi(r)}\int_0^r \rme^{-2\xi(\rho)}\rho^{2m-1} \,\rmd\rho & \text{if}\ m\geqslant 1\,,\\
r^{-m} \rme^{\xi(r)} & \text{if}\ m\leqslant 0\,.
\end{cases}
$$
Here $\xi$ is the function from \eqref{xi}, which is radial and which by Newton's theorem \cite[Thm. 9.2]{LL} can be rewritten as
\begin{equation}\label{radial xi}
\xi(r) = -\int_0^r B(\rho)\, \rho \,\rmd\rho\,\ln r- \int_r^\infty B(\rho)\, \rho \ln \rho\,\rmd\rho, \quad r> 0\,.
\end{equation}
The functions $\omega^\pm_m$ are important since they solve
\begin{equation}\label{eq:vleq}
-r^{-1} \big(r (\omega^\pm_m)'\big)' + W_m^\pm \omega^\pm_m =0
\end{equation}
and belong locally to the form domain of $T_m^\pm$. Moreover, using that $\xi(r)=-\Phi \ln r$ for $r\geqslant R$ and that $|\xi(r)|$ is bounded for $0\leqslant r\leqslant R$, and assuming $\Phi\geqslant 0$ for the sake of definiteness one easily finds the two-sided bounds (with constants independent of $m$ and $r$)
\begin{align}
\omega_m^-(r) & \asymp
\begin{cases}\label{eq:omega-}
m^{-1} r^m (r+R)^\Phi \,, & m>0 \,, \\
r^{|m|} (r+R)^{-\Phi} \,, & m\leqslant 0 \,,
\end{cases} \\
\omega_m^+(r) & \asymp
\begin{cases}\label{eq:omega+}
r^m (r+R)^\Phi \,, & m\geqslant 0 \,, \\
|m|^{-1} r^{|m|} (r+R)^{\Phi-2|m|} \,, & -\Phi<m< 0 \,, \\
|m|^{-1} \Big(\dfrac{r}{r+R}\Big)^{|m|} \big(1 + \ln_+ (r/R) \big) \,, & m=-\Phi<0 \,, \\
|m|^{-1} r^{|m|} (r+R)^{-\Phi} \,, & m< -\Phi \,.
\end{cases}
\end{align}


\subsection{Proof of Theorems \ref{main} and \ref{mainint}}

Because of \eqref{eq:assbv} $B_\pm$ and $V$ are form compact with respect to $-\Delta$ and hence by the diamagnetic inequality also with respect to $(D+A)^2$. (A proof of this fact may be based on Proposition \ref{hardy} and the Sobolev embedding theorem.) Hence $(W_m^\pm)_-$ is form compact with respect to $-r^{-1}\partial_r r \partial_r$ and $V$ is form compact with respect to $T_m^\pm$, so Assumption \ref{ass:w} is satisfied.

First assume that $\Phi>0$. We claim that for all sufficiently small $\alpha>0$ the operators $T_m^--\alpha V$ with $-\Phi\leqslant m\leqslant 0$ have a unique negative eigenvalue $\lambda(T_m^--\alpha V)$. Indeed, this follows from Proposition \ref{vl} since for these values of $m$ the functions $\omega^-_m$ are positive, decay like $r^{-\Phi-m}$, belong locally to the form domain and satisfy \eqref{eq:vleq}. Moreover, putting
$$
v_m^\pm := \int_0^\infty V(r) \omega_m^\pm(r)^2 r\,\rmd r
$$
Theorem \ref{main1d} yields the following asymptotic behavior as $\alpha\to 0+$.
If $-\Phi+1<m\leqslant 0$, then
\begin{align*}
\lambda(T_m^--\alpha V)  = - \frac{\alpha v_m^-}{\int_0^\infty \omega_m^-{}^2 r \,\rmd r} \left(1+ 
\begin{cases} 
O\left(\alpha\right) \,,& -\Phi+2<m\leqslant 0 \\
O\big( \alpha|\ln\alpha| \big) \,,& m=-\Phi+2\leqslant 0 \\
O\left(\alpha^\mu\right) \,,& -\Phi+1< m < -\Phi +2
\end{cases}
\right) \,.
\end{align*}
If $m=-\Phi+1\leq 0$, then
\begin{equation*}
\lambda(T_{m}^--\alpha V)  = - \frac{2\alpha v^-_{m}}{|\ln\alpha|} \bigg(1 + O\Big(  \frac{\ln|\ln\alpha |}{|\ln\alpha |}\Big)\bigg) \,.
\end{equation*}
If $-\Phi<m<-\Phi+1$, then with $c_\mu$ from \eqref{c_nu}
\begin{align*}
\lambda(T_m^--\alpha V) & =
- c_\mu \left(\alpha v^-_m\right)^\frac1\mu\big(1+ O(\alpha^{\min\{1, \frac1\mu- 1\}})\big) \,.
\end{align*}
If $m=-\Phi$, then
\begin{align*}
\lambda(T_{m}^--\alpha V) & = -\exp\Big(- \frac2{\alpha v^-_{m}} \big(1+ O(\alpha) \big)\Big) \,.
\end{align*}
These asymptotics coincide with those claimed in Theorems \ref{main} and \ref{mainint} since for $k\in\N_0$ one has
$$
\Omega_{k}^\pm(x) = (2\pi)^{-1/2} \omega_{\pm k}^\pm (r) \rme^{\pm ik\phi}
$$
and $v_k = v_{\pm k}^\mp$ if $\pm\Phi>0$.

In order to complete the proof of Theorems \ref{main} and \ref{mainint} in the case $\Phi>0$ we need to show that there exists an $\alpha_c>0$ such that for all $0<\alpha\leqslant\alpha_c$ the operators $T_m^--\alpha V$ with $m<-\Phi$ and $m>0$, as well as the operators $T^+_m-\alpha V$, $m\in\Z$, are non-negative. Note that Proposition \ref{vl} shows that these operators have no negative eigenvalues for $\alpha>0$ small (since $\omega^\pm_m$ is unbounded), but it does \emph{not} imply Theorems \ref{main} and \ref{mainint} since it gives no uniformity in $m$.

Instead, we will deduce the assertion from Proposition \ref{hardy} by showing that
$$
\sup_{m\in\Z} \sup_{0<r<\infty} \int_0^r V (\omega_m^+)^2 \rho\,\rmd\rho \int_r^\infty (\omega_m^+)^{-2} \rho^{-1}\,\rmd\rho < \infty
$$
and similarly with $\omega_m^+$ replaced by $\omega_m^-$ and the supremum restricted to $m<-\Phi$ and $m>0$. According to the two-sided estimates \eqref{eq:omega-} and \eqref{eq:omega+} on $\omega_m^\pm$ this is equivalent to showing that
$$
\sup_{k\in\N} \sup_{0<r\leqslant R} k^{-1} r^{-2k} \int_0^r V \rho^{2k+1} \,\rmd\rho < \infty
\quad\text{and}\quad
\sup_{0<r\leqslant R} |\ln r| \int_0^r V \rho \,\rmd\rho <\infty \,.
$$
These estimates are easily deduced from \eqref{eq:assbv}. This completes the proof in the case $\Phi>0$.

The proof in the case $\Phi=0$ is similar, but now both $\omega^+_0$ and $\omega^-_0$ are bounded positive solutions which locally belong to the form domain. The non-negativity for $m\neq 0$ follows again by Proposition \ref{hardy}.

The result for $\Phi<0$ follows from that for $\Phi>0$ since complex conjugation is an anti-unitary operator which switches the sign of $B$. This completes the proof of Theorems \ref{main} and \ref{mainint}.


\bibliographystyle{amsalpha}

\end{document}